\documentclass[12pt]{amsart}
\usepackage{amsfonts,amssymb,amscd,amsmath,enumerate,verbatim}
\def\NZQ{\mathbb}               

\def\QQ{{\NZQ Q}}
\def\ZZ{{\NZQ Z}}
\def\RR{{\NZQ R}}

\def\frk{\mathfrak}               

\def\Phi{{\frk N}}
\def\opn#1#2{\def#1{\operatorname{#2}}} 
\opn\chara{char} 
\opn\length{\ell} 
\opn\pd{pd} 
\opn\rk{rk}
\opn\projdim{proj\,dim} 
\opn\injdim{inj\,dim} 
\opn\rank{rank}
\opn\depth{depth} 
\opn\grade{grade} 
\opn\height{height}
\opn\embdim{emb\,dim} 
\opn\codim{codim}

\opn\Tr{Tr} 
\opn\bigrank{big\,rank}
\opn\superheight{superheight}
\opn\lcm{lcm}
\opn\trdeg{tr\,deg}
\opn\reg{reg} 
\opn\lreg{lreg} 
\opn\ini{in} 
\opn\lpd{lpd}
\opn\size{size}
\opn\mult{mult}
\opn\dist{dist}
\opn\cone{cone}
\opn\lex{lex}
\opn\rev{rev}

\opn\div{div} \opn\Div{Div} \opn\cl{cl} \opn\Cl{Cl}
\opn\Spec{Spec} \opn\Supp{Supp} \opn\supp{supp} \opn\Sing{Sing}
\opn\Ass{Ass} \opn\Min{Min}
\opn\Ann{Ann} \opn\Rad{Rad} \opn\Soc{Soc}
\opn\Syz{Syz} \opn\Im{Im} \opn\Ker{Ker} \opn\Coker{Coker}
\opn\Am{Am} \opn\Hom{Hom} \opn\Tor{Tor} \opn\Ext{Ext}
\opn\End{End} \opn\Aut{Aut} \opn\id{id} \opn\ini{in}

\opn\nat{nat}
\opn\pff{pf}
\opn\Pf{Pf} \opn\GL{GL} \opn\SL{SL} \opn\mod{mod} \opn\ord{ord}
\opn\Gin{Gin}
\opn\Hilb{Hilb}\opn\adeg{adeg}\opn\std{std}\opn\ip{infpt}
\opn\Pol{Pol}
\opn\sat{sat}
\opn\Var{Var}
\opn\Gen{Gen}

\opn\aff{aff} \opn\con{conv} \opn\inte{int} \opn\st{st}
\opn\lk{lk} \opn\cn{cn} \opn\core{core} \opn\vol{vol}
\opn\link{link} \opn\star{star}
\opn\gr{gr}

\def\Ac{{\mathcal A}}
\def\Bc{{\mathcal B}}

\def\Pc{{\mathcal P}}

\def\Qc{{\mathcal Q}}
\def\Wc{{\mathcal W}}
\def\pot#1#2{#1[\kern-0.28ex[#2]\kern-0.28ex]}

\opn\dirlim{\underrightarrow{\lim}}
\opn\inivlim{\underleftarrow{\lim}}
\let\to=\rightarrow

\def\Implies{\ifmmode\Longrightarrow \else
        \unskip${}\Longrightarrow{}$\ignorespaces\fi}
\def\implies{\ifmmode\Rightarrow \else
        \unskip${}\Rightarrow{}$\ignorespaces\fi}
\def\iff{\ifmmode\Longleftrightarrow \else
        \unskip${}\Longleftrightarrow{}$\ignorespaces\fi}

\let\:=\colon
\newtheorem{Theorem}{Theorem}[section]
\newtheorem{Lemma}[Theorem]{Lemma}
\newtheorem{Corollary}[Theorem]{Corollary}
\newtheorem{Proposition}[Theorem]{Proposition}

\newtheorem{Example}[Theorem]{Example}

\let\epsilon\varepsilon
\let\phi=\varphi
\let\kappa=\varkappa
\def\qed{\ifhmode\textqed\fi
      \ifmmode\ifinner\quad\qedsymbol\else\dispqed\fi\fi}
\def\textqed{\unskip\nobreak\penalty50
       \hskip2em\hbox{}\nobreak\hfil\qedsymbol
       \parfillskip=0pt \finalhyphendemerits=0}
\def\dispqed{\rlap{\qquad\qedsymbol}}

\opn\dis{dis}
\opn\height{height}
\opn\dist{dist}
\def\pnt{{\raise0.5mm\hbox{\large\bf.}}}

\opn\Lex{Lex}

\begin{document}
\title{Integer decomposition property of free sums of convex polytopes}
\author{Takayuki Hibi and Akihiro Higashitani}
\address{Takayuki Hibi,
Department of Pure and Applied Mathematics,
Graduate School of Information Science and Technology,
Osaka University,
Toyonaka, Osaka 560-0043, Japan}
\email{hibi@math.sci.osaka-u.ac.jp}
\address{Akihiro Higashitani,
Department of Mathematics, Graduate School of Science, 
Kyoto University, Kitashirakawa-Oiwake cho, Sakyo-ku, Kyoto 606-8502, Japan}
\email{ahigashi@math.kyoto-u.ac.jp}
\begin{abstract}
Let $\Pc \subset \RR^{d}$ and $\Qc \subset \RR^e$ be integral convex polytopes 
of dimension $d$ and $e$ which contain the origin of $\RR^{d}$ and $\RR^e$, respectively. 
In the present paper, under some assumptions, 
the necessary and sufficient condition for the free sum of $\Pc$ and $\Qc$ 
to possess the integer decomposition property will be presented. 
\end{abstract}
\maketitle
\section*{Introduction}
%
%
%
A convex polytope is called {\em integral} if any of its vertices has integer coordinates. 
Let $\Pc \subset \RR^{d}$ and $\Qc \subset \RR^{e}$ be convex polytopes 
and suppose that ${\bf 0}_d \in \Pc$ and ${\bf 0}_e \in \Qc$, 
where ${\bf 0}_d \in \RR^d$ denotes the origin of $\RR^d$ and ${\bf 0}_e \in \RR^e$ denotes that of $\RR^e$. 
We introduce the canonical injections $\mu : \RR^{d} \to \RR^{d+e}$ 
by setting $\mu(\alpha) = (\alpha, {\bf 0}_e) \in \RR^{d+e}$ 
with $\alpha \in \RR^{d}$ and $\nu : \RR^{e} \to \RR^{d+e}$ 
by setting $\nu(\beta) = ({\bf 0}_d, \beta) \in \RR^{d+e}$ with $\beta \in \RR^{e}$. 
In particular, $\mu({\bf 0}_d) = \nu({\bf 0}_e)={\bf 0}_{d+e}$, 
where ${\bf 0}_{d+e}$ denotes the origin of $\RR^{d+e}$. 
Then $\mu(\Pc)$ and $\nu(\Qc)$ are convex polytopes of $\RR^{d+e}$ 
with $\mu(\Pc) \cap \nu(\Qc) = {\bf 0}_{d+e} \in \RR^{d+e}$. 
The {\em free sum} of $\Pc$ and $\Qc$ 
is the convex hull of the set $\mu(\Pc) \cup \nu(\Qc)$ in $\RR^{d+e}$.  
It is written as $\Pc \oplus \Qc$. One has $\dim (\Pc \oplus \Qc) = \dim \Pc + \dim \Qc$.

For a convex polytope $\Pc \subset \RR^{d}$ and for each integer $n \geq 1$, 
we write $n \Pc$ for the convex polytope 
$\{ n \alpha \, : \, \alpha \in \Pc \} \subset \RR^{d}$.
We say that an integral convex polytope $\Pc \subset \RR^{d}$
possesses the {\em integer decomposition property} if, for each $n \geq 1$
and for each $\gamma \in n \Pc \cap \ZZ^{d}$, 
there exist $\gamma^{(1)}, \ldots, \gamma^{(n)}$ belonging to $\Pc \cap \ZZ^{d}$
such that $\gamma = \gamma^{(1)} + \ldots + \gamma^{(n)}$.

Let $\Pc \subset \RR^{d}$ and $\Qc \subset \RR^{e}$ be convex polytopes 
containing the origin (of $\RR^{d}$ or $\RR^{e}$). 
It is then easy to see that if the free sum of $\Pc$ and $\Qc$ possesses the integer decomposition property, 
then each of $\Pc$ and $\Qc$ possesses the integer decomposition property. 
On the other hand, the converse is not true in general. (See Example \ref{idp}.) 
The purpose of the present paper is to show the following


\begin{Theorem}\label{Stockholm}
Let $\Pc \subset \RR^d$ and $\Qc \subset \RR^e$ be integral convex polytopes 
of dimension $d$ and dimension $e$ containing ${\bf 0}_d$ and ${\bf 0}_e$, respectively. 
Suppose that $\Pc$ and $\Qc$ satisfy $\ZZ(\Pc \cap \ZZ^d) = \ZZ^d$, $\ZZ(\Qc \cap \ZZ^e) = \ZZ^e$ and 
\begin{align}\label{condition}
(\Pc \oplus \Qc) \cap \ZZ^{d+e} = \mu( \Pc \cap \ZZ^d) \cup \nu( \Qc \cap \ZZ^e).
\end{align}
Then the free sum $\Pc \oplus \Qc$ possesses the integer decomposition property 
if and only if the following two conditions are satisfied: 
\begin{itemize}
\item each of $\Pc$ and $\Qc$ possesses the integer decomposition property; 
\item either $\Pc$ or $\Qc$ satisfies that 
the equation of each facet is of the form $\sum_{i=1}^{f} a_{i}z_{i} = b$, 
where each $a_{i}$ is an integer, $b \in \{ 0, 1\}$ and $f \in \{d,e\}$. 
\end{itemize}
\end{Theorem}

An integral convex polytope $\Pc \subset \RR^d$ is called a {\em $(0, 1)$-polytope} 
if each vertex of $\Pc$ belongs to $\{0, 1\}^{d}$.  
It then follows that the equality \eqref{condition} is always satisfied 
if each of $\Pc$ and $\Qc$ is a $(0,1)$-polytope. 
As an immediate corollary of Theorem \ref{Stockholm}, we also obtain the following 
\begin{Corollary}\label{kei}
Let $\Pc \subset \RR^d$ be a $(0,1)$-polytope of dimension $d$ containing ${\bf 0}_d$ 
and $\Qc \subset \RR^e$ an integral convex polytope of dimension $e$ containing ${\bf 0}_e$. 
Suppose that $\Pc$ and $\Qc$ satisfy $\ZZ(\Pc \cap \ZZ^d) = \ZZ^d$ and $\ZZ(\Qc \cap \ZZ^e) = \ZZ^e$. 
Then the free sum $\Pc \oplus \Qc$ possesses the integer decomposition property 
if and only if the following two conditions are satisfied: 
\begin{itemize}
\item each of $\Pc$ and $\Qc$ possesses the integer decomposition property; 
\item either $\Pc$ or $\Qc$ satisfies that 
the equation of each facet is of the form $\sum_{i=1}^{f} a_{i}z_{i} = b$, 
where each $a_{i}$ is an integer, $b \in \{ 0, 1\}$ and $f \in \{d,e\}$. 
\end{itemize}
\end{Corollary}

\begin{Example}
\label{idp}
{\em
Even though $\Pc$ and $\Qc$ possess the integer decomposition property, 
the free sum $\Pc \oplus \Qc$ may fail to possess the integer decomposition property.
For example, let $\Pc \subset \RR^{3}$ be the $(0, 1)$-polytope with the vertices 
$(0,0,0), (1,1,0)$, $(1,0,1), (0,1,1)$ and $(1,0,0)$. 
Then $\Pc$ possesses the integer decomposition property, but the free sum $\Pc \oplus \Pc$ 
fails to possess the integer decomposition property. 
In fact, $z_1+z_2+z_3=2$ is the equation of a facet of $\Pc$. 
}
\end{Example}

A structure of the present paper is as follows. 
In Section \ref{jouken}, we will consider the condition for $\Pc$ and $\Qc$ 
to satisfy the equality \eqref{condition}. 
In Section \ref{proof}, a proof of Theorem \ref{Stockholm} will be given.


\section{When does the equality \eqref{condition} hold?}\label{jouken}

Let $V(\Pc)$ be the set of vertices of $\Pc$ and let $V(\Qc)$ be that of $\Qc$. 
First, for $W \subset V(\Pc) \setminus \{{\bf 0}_d\}$, let 
$$\inte(W) = (\con( W \cup \{{\bf 0}_d\}) \setminus \partial \con( W \cup \{{\bf 0}_d\})) \cap \ZZ^d.$$ 
For $W \subset V(\Qc) \setminus \{{\bf 0}_e\}$, $\inte(W)$ is also defined in the same way. 
Next, we define 
$$\Wc(\Pc)=\left\{ W \subset V(\Pc) \setminus \{{\bf 0}_d\} : W \text{ is linearly independent and } \inte(W) \not= \emptyset \right\}.$$
In the similar way, we also define $\Wc(\Qc)$. 
Last, for any $W=\{w_1,\ldots,w_m\} \in \Wc(\Pc)$ (similarly, for any $W \in \Wc(\Qc)$), let 
$$\min(W)=\min\left\{ \sum_{i=1}^m r_i : \sum_{i=1}^m r_iw_i \in \inte(W)\right\}.$$ 
Then $0<\min(W)<1$.

\begin{Proposition}\label{meidai}
Let $\Pc \subset \RR^d$ and $\Qc \subset \RR^e$ be integral convex polytopes 
containing ${\bf 0}_d$ and ${\bf 0}_e$, respectively. 
Then the free sum $\Pc \oplus \Qc$ satisfies the equality \eqref{condition}
if and only if 
\begin{itemize}
\item $\Wc(\Pc)=\emptyset$ or $\Wc(\Qc)=\emptyset$, or 
\item $\Wc(\Pc)\not=\emptyset$, $\Wc(\Qc)\not=\emptyset$ and 
$\min(F) + \min(G) > 1$ for any $F \in \Wc(\Pc)$ and $G \in \Wc(\Qc)$. 
\end{itemize}
\end{Proposition}
\begin{proof}
{\bf ``Only if''} Assume that there exist $F \in \Wc(\Pc)$ and $G \in \Wc(\Qc)$ such that $\min(F) + \min(G) \leq 1$. 
Then each of $F$ and $G$ is linearly independent. Let $F=\{v_1,\ldots,v_n\}$ and let $G=\{w_1,\ldots,w_m\}$. 
Then there are $0<r_1,\ldots,r_n<1$, $0<s_1,\ldots,s_m<1$ such that 
$\sum_{i=1}^n r_i v_i \in \inte(F)$ and $\sum_{i=1}^m s_i w_i \in \inte(G)$, 
where $0<\sum_{i=1}^n r_i<1$ and $0<\sum_{i=1}^m s_i<1$ with $\sum_{i=1}^n r_i + \sum_{i=1}^m s_i \leq 1$. 
Let us consider $$\alpha=\sum_{i=1}^n r_i \mu(v_i)+\sum_{i=1}^m s_i \nu(w_i) 
\in \RR^{d+e}.$$
Since $\sum_{i=1}^n r_i v_i \in \ZZ^d$, we have 
$\sum_{i=1}^n r_i \mu(v_i) \in \ZZ^{d+e}$. Similarly, $\sum_{i=1}^m s_i \nu(w_i) \in \ZZ^{d+e}$. 
Thus, $\alpha \in \ZZ^{d+e}$. Moreover, since $\sum_{i=1}^n r_i + \sum_{i=1}^m s_i \leq 1$, 
we have $\alpha \in \Pc \oplus \Qc$. Hence, $\alpha \in (\Pc \oplus \Qc) \cap \ZZ^{d+e}$. On the other hand, 
since $\sum_{i=1}^n r_i v_i \not= {\bf 0}_d$ and $\sum_{i=1}^m s_i w_i \not= {\bf 0}_e$, 
we see that $\alpha \not\in \mu ( \Pc \cap \ZZ^d ) \cup \nu ( \Qc \cap \ZZ^e )$. 
These mean that the equality \eqref{condition} is not satisfied. 

\noindent
{\bf ``If''} Assume that \eqref{condition} is not satisfied. 
Since the inclusion $(\Pc \oplus \Qc) \cap \ZZ^{d+e} \supset \mu ( \Pc \cap \ZZ^d ) \cup \nu ( \Qc \cap \ZZ^e )$ 
is always satisfied, we may assume that there is $\alpha$ 
belonging to $(\Pc \oplus \Qc) \cap \ZZ^{d+e} \setminus \mu ( \Pc \cap \ZZ^d ) \cup \nu ( \Qc \cap \ZZ^e )$. 
Then $\alpha$ can be written like $$\alpha=\sum_{i=1}^n r_i \mu(v_i)+\sum_{i=1}^m s_i \nu(w_i),$$ 
where $v_1,\ldots,v_n \in V(\Pc)\setminus \{{\bf 0}_d\}$, $w_1,\ldots,w_m \in V(\Qc)\setminus \{{\bf 0}_e\}$, 
$0 \leq r_1,\ldots,r_n \leq 1$, $ 0 \leq s_1,\ldots,s_m \leq 1$ and $\sum_{i=1}^nr_i+\sum_{i=1}^ms_i \leq 1$. 
By Carath\'eodory's Theorem (cf. \cite[Corollary 7.1i]{Sch}), we can choose $\mu(v_1),\ldots,\mu(v_n),\nu(w_1),\ldots,\nu(w_m)$ 
as linearly independent vectors of $\RR^{d+e}$, that is, 
$v_1,\ldots,v_n$ are linearly independent in $\RR^d$ and so are $w_1,\ldots,w_m$ in $\RR^e$. 
Moreover, if $\sum_{i=1}^nr_i=0$, then $\alpha \in \nu(\Qc \cap \ZZ^e)$, a contradiction. Similarly, 
if $\sum_{i=1}^ms_i=0$, then $\alpha \in \mu(\Pc \cap \ZZ^e)$, a contradiction. Thus, 
we also assume $\sum_{i=1}^nr_i>0$ and $\sum_{i=1}^ms_i>0$.

We consider $v = \sum_{i=1}^n r_iv_i \in \ZZ^d$. 
Since $\sum_{i=1}^nr_i>0, \sum_{i=1}^ms_i>0$ and $\sum_{i=1}^nr_i+\sum_{i=1}^ms_i \leq 1$, we have $0<\sum_{i=1}^n r_i<1$. 
Thus, $v \in \Pc \cap \ZZ^d$. Let $v_{i_1},\ldots,v_{i_g}$ be all of $v_i$'s such that $r_i > 0$ 
and let $S=\{v_{i_1}, \ldots, v_{i_g}\}$. Then $S$ is also linearly independent and $v \in \inte(S)$. Hence, $S \in \Wc(\Pc)$. 
Similarly, let $w_{j_1},\ldots,w_{j_h}$ be all of $w_i$'s such that $s_i > 0$ 
and let $T=\{w_{j_1}, \ldots, w_{j_h}\}$. Then $T \in \Wc(\Qc)$. Now we see 
$$\min(S)+\min(T) \leq \sum_{k=1}^g r_{i_k} + \sum_{k=1}^h s_{j_k} = \sum_{i=1}^n r_i+\sum_{i=1}^m s_i \leq 1,$$ 
as required. 
\end{proof}

\begin{Example}{\em 
(a) Let $\Pc \subset \RR^d$ be a $(0,1)$-polytope. Then we easily see that $\Wc(\Pc)=\emptyset$. 
Thus, if $\Pc$ or $\Qc$ is a $(0,1)$-polytope in Proposition \ref{meidai}, then the equality \eqref{condition} always holds. 

(b) Let $\Pc=\con(\{(0,0),(1,0),(1,2)\}) \subset \RR^2$ and let $\Qc=\con(\{0,2\}) \subset \RR^1$. 
Then $\Wc(\Qc)\not=\emptyset$ but $\Wc(\Pc)=\emptyset$. Thus the equality \eqref{condition} holds. 

(c) Let $\Pc = \Qc = \con(\{(0,0),(2,1),(1,2)\}) \subset \RR^2$ and consider $W=\{(2,1),(1,2)\}$. 
Then we see that $\Wc(\Pc)=\{W\}$. On the other hand, we also have $\min(W)=2/3$. Thus the equality \eqref{condition} holds. 
}\end{Example}

\section{A proof of Theorem \ref{Stockholm}}\label{proof}

Let $\Pc \subset \RR^{d}$ be an integral convex polytope of dimension $d$. 
A {\em configuration} arising from $\Pc$ is the finite set 
$\Ac = \{ (\alpha, 1) \in \ZZ^{d+1} : \alpha \in \Pc \cap \ZZ^{d} \}$. 
We say that $\Ac$ is {\em normal} if 
\[\ZZ_{\geq 0} \Ac = \ZZ \Ac \cap \QQ_{\geq 0} \Ac,\] 
where $\ZZ_{\geq 0}$ is the set of nonnegative integers and
$\QQ_{\geq 0}$ is the set of nonnegative rational numbers.

Recall from \cite[Chapter~IX]{HibiRedBook} 
what the Ehrhart polynomial of an integral convex polytope 
is.
Let $\Pc \subset \RR^{d}$ be an integral convex polytope of dimension $d$
and, for each integer $n \geq 1$, 
write $i(\Pc, n)$ for the number of integer points belonging to $n\Pc$,
i.e., $i(\Pc, n) = |n\Pc \cap \ZZ^{d}|$.
It is known that $i(\Pc, n)$ is a polynomial in $n$ of degree $d$
with $i(\Pc, 0) = 1$.
We call $i(\Pc, n)$ the {\em Ehrhart polynomial} of $\Pc$.  We then define
the integers $\delta_{0}, \delta_{1}, \delta_{2}, \ldots$ by the formula
\[
(1 - \lambda)^{d+1} \big[1 + \sum_{n=1}^{\infty} i(\Pc, n) \lambda^{n}\big]
= \sum_{n=0}^{\infty} \delta_{n} \lambda^{n}.
\]
It then follows that $\delta_{n} = 0$ for $n > d$.  The polynomial
\[
\delta(\Pc) = \sum_{n=0}^{d} \delta_{n} \lambda^{n}
\]
is called the {\em $\delta$-polynomial} of $\Pc$.

Let $K[t_{1}, t_{1}^{-1}, \ldots, t_{d}, t_{d}^{-1}, s]$ denote the Laurent polynomial
ring in $d + 1$ variables over a field $K$. 
If $\alpha = (\alpha_{1}, \ldots, \alpha_{d}) \in \Pc \cap \ZZ^{d}$, then
we write $u_{\alpha}$ for the Laurent monomial
$t_{1}^{\alpha_{1}} \cdots t_{d}^{\alpha_{d}} \in 
K[t_{1}, t_{1}^{-1}, \ldots, t_{d}, t_{d}^{-1}]$. 
The {\em toric ring} of $\Ac$ is the subring $K[\Ac]$ of 
$K[t_{1}, t_{1}^{-1}, \ldots, t_{d}, t_{d}^{-1}, s]$
which is generated by those Laurent monomials $u_{\alpha}s$
with $\alpha \in \Pc \cap \ZZ^{d}$.
Let $K[\{ x_{\alpha} \}_{\alpha \in \Pc \cap \ZZ^{d}}]$ be
the polynomial ring in $|\Pc \cap \ZZ^{d}|$ variables over $K$
with each $\deg x_{\alpha} = 1$.
We then define the surjective ring homomorphism 
$\pi : K[\{ x_{\alpha} \}_{\alpha \in \Pc \cap \ZZ^{d}}] \to K[\Ac]$
by setting $\pi(x_{\alpha}) = u_{\alpha}s$ for each $\alpha \in \Pc \cap \ZZ^{d}$.

Finally, the Hilbert function of 
the toric ring $K[\Ac]$ of the configuration $\Ac$ arising from
an integral convex polytope $\Pc \subset \RR^{d}$ of dimension $d$
is introduced. 
We write $(K[\Ac])_{n}$ for the subspace of $K[\Ac]$ spanned by those Laurent
monomials of the form 
\[
(u_{\alpha^{(1)}}s)(u_{\alpha^{(2)}}s) \cdots (u_{\alpha^{(n)}}s)
\]
with each $\alpha^{(i)} \in \Pc \cap \ZZ^{d}$.
In particular $(K[\Ac])_{0} = K$ and 
$(K[\Ac])_{1} = \sum_{\alpha \in \Pc \cap \ZZ^{d}}K u_{\alpha}s$.
The {\em Hilbert
function} of $K[\Ac]$ is the numerical function
\[
H(K[\Ac], n) = \dim_{K} (K[\Ac])_{n}, \, \, \, \, \, n = 0, 1, 2, \ldots.
\] 
Thus in particular $H(K[\Ac], 0) = 1$ and $H(K[\Ac], 1) = |\Pc \cap \ZZ^{d}|$.
We then define
the integers $h_{0}, h_{1}, h_{2}, \ldots$ by the formula
\[
(1 - \lambda)^{d+1} \big[\sum_{n=0}^{\infty} H(K[\Ac], n) \lambda^{n}\big]
= \sum_{n=0}^{\infty} h_{n} \lambda^{n}.
\]
A basic fact \cite[Theorem 11.1]{AM} of Hilbert functions 
guarantees that $h_{n} = 0$ for $n \gg 0$.  
We say that the polynomial
\[
h(K[\Ac]) = \sum_{n=0}^{\infty} h_{n} \lambda^{n}
\]
is the {\em $h$-polynomial} of $K[\Ac]$.

\begin{Lemma}
\label{normal}
Let $\Pc \subset \RR^{d}$ be an integral convex polytope of dimension $d$ and 
$\Ac \subset \ZZ^{d+1}$ the configuration arising from $\Pc$. 
Suppose that $\Pc$ satisfies $\ZZ(\Pc \cap \ZZ^d) = \ZZ^d$. Then the following conditions are equivalent:
\begin{enumerate}
\item[(i)] $\Pc$ possesses the integer decomposition property;
\item[(ii)] $\Ac$ is normal;
\item[(iii)] $\delta(\Pc) = h(K[\Ac])$.
\end{enumerate}
\end{Lemma}

\begin{proof}
It follows that $\Pc$ possesses the integer decomposition property 
if and only if, for $\alpha \in n \Pc \cap \ZZ^{d}$, 
one has $(\alpha, n) \in \ZZ_{\geq 0} \Ac$. 
Since $\ZZ(\Pc \cap \ZZ^d) = \ZZ^d$, i.e., $\ZZ \Ac = \ZZ^{d+1}$, 
it follows that $\Ac$ is normal if and only if $\ZZ_{\geq 0} \Ac = \ZZ^{d+1} \cap \QQ_{\geq 0} \Ac$.
Moreover, for $\alpha \in \QQ^{d}$, one has $\alpha \in n \Pc$ 
if and only if $(\alpha, n) \in \QQ_{\geq 0}\Ac$.
Hence (i) $\Leftrightarrow$ (ii) follows.

In general, one has $i(\Pc, n) \geq H(K[\Ac], n)$ for $n \in \ZZ_{\geq 0}$.
Furthermore, it follows that 
$i(\Pc, n) = H(K[\Ac], n)$ for all $n \in \ZZ_{\geq 0}$ if and only if 
$\Pc$ possesses the integer decomposition property.  
Hence (i) $\Leftrightarrow$ (iii) follows.
\, \, \, \, \, \, \, \, \, \, \, \, \, \, \, \, \, \, \,
\end{proof}

\begin{Lemma}
\label{Sapporo}
Let $\Pc \subset \RR^{d}$ and $\Qc \subset \RR^e$ be integral convex polytopes 
of dimension $d$ and $e$ which contain the origin of $\RR^{d}$ and $\RR^e$, respectively. 
Let $\Ac \subset \ZZ^{d+1}$ and $\Bc \subset \ZZ^{e+1}$ be the configurations 
arising from $\Pc$ and $\Qc$, respectively. 
Let $\Ac \oplus \Bc \subset \ZZ^{d+e+1}$ denote the configuration arising from 
the free sum $\Pc \oplus \Qc \subset \RR^{d+e}$.  Suppose that 
\begin{eqnarray*}
(\Pc \oplus \Qc) \cap \ZZ^{d+e} = \mu(\Pc \cap \ZZ^{d}) \cup \nu(\Qc \cap \ZZ^{e}).
\end{eqnarray*}
Then
\[h(K[\Ac \oplus \Bc]) = h(K[\Ac])h(K[\Bc]).\]
Furthermore, if $\Pc \oplus \Qc$ possesses the integer decomposition property,
then 
\[
\delta(\Pc \oplus \Qc) = \delta(\Pc)\delta(\Qc).
\]
\end{Lemma}

\begin{proof}
Let $K[\Ac] \subset K[t_{1}, t_{1}^{-1}, \ldots, t_{d}, t_{d}^{-1}, s]$
and $K[\Bc] \subset K[{t'}_{1}, {t'}_{1}^{-1}, \ldots, {t'}_{e}, {t'}_{e}^{-1}, s']$.
Then $K[\Ac \oplus \Bc] = (K[\Ac] \otimes K[\Bc])/(s - s')$.
Hence 
$h(K[\Ac \oplus \Bc]) =  
h(K[\Ac] \otimes K[\Bc]) = h(K[\Ac])h(K[\Bc])$, as desired. 

If, furthermore, $\Pc \oplus \Qc$ possesses the integer decomposition property,
then each of $\Pc$ and $\Qc$ possesses the integer decomposition property.
Lemma \ref{normal} then says that 
$\delta(\Pc \oplus \Qc) = h(K[\Ac \oplus \Bc])$, 
$\delta(\Pc) = h(K[\Ac])$ and
$\delta(\Qc) = h(K[\Bc])$.  
Hence $\delta(\Pc \oplus \Qc) = \delta(\Pc)\delta(\Qc)$, as required.
\, \, \, \, \, \, \, \, \, \, \, \, \, \, \, \, \, \, \, \, \, \, \, \, \, \, \, \, \, \, 
\end{proof}

We also recall the following theorem. 

\begin{Theorem}[{\cite[Theorem 1.4]{BJM}}]\label{facets}
Let $\Pc \subset \RR^d$ and $\Qc \subset \RR^e$ be integral convex polytopes 
containing the origin (of $\RR^d$ or $\RR^e$). 
Then the equality $\delta(\Pc \oplus \Qc) = \delta(\Pc) \delta(\Qc)$ holds 
if and only if either $\Pc$ or $\Qc$ satisfies that 
the equation of each facet is of the form $\sum_{i=1}^f a_{i}z_{i} = b$, 
where each $a_{i}$ is an integer, $b \in \{ 0, 1\}$ and $f \in \{d,e\}$. 
\end{Theorem}

We are now in the position to give a proof of Theorem \ref{Stockholm}.

\begin{proof}[Proof of Theorem \ref{Stockholm}]
Assume that each of $\Pc$ and $\Qc$ possesses the integer decomposition property and 
either $\Pc$ or $\Qc$ satisfies the condition on its facets described in Theorem \ref{Stockholm}. 
It then follows from Theorem \ref{facets} that the condition on the facets is equivalent to satisfying that 
\begin{align}\label{111}
\delta(\Pc \oplus \Qc) = \delta(\Pc) \delta(\Qc). 
\end{align}
Moreover, since each of $\Pc$ and $\Qc$ possesses the integer decomposition property, 
we have the equalities $\delta(\Pc)=h(K[\Ac])$ and $\delta(\Qc)=h(K[\Bc])$ by Lemma \ref{normal}. 
In particular, one has 
\begin{align}\label{222}
\delta(\Pc) \delta(\Qc)=h(K[\Ac])h(K[\Bc]). 
\end{align}
Furthermore, since the equality \eqref{condition} is satisfied, it follows from Lemma \ref{Sapporo} that 
\begin{align}\label{333}
h(K[\Ac \oplus \Bc])=h(K[\Ac])h(K[\Bc]), 
\end{align}
where $\Ac \oplus \Bc \subset \ZZ^{d+e+1}$ denotes the configuration arising from $\Pc \oplus \Qc \subset \RR^{d+e}$. 
Hence, by \eqref{111}, \eqref{222} and \eqref{333}, we obtain $$\delta(\Pc \oplus \Qc)=h(K[\Ac \oplus \Bc]).$$ 
Therefore, from Lemma \ref{normal}, we conclude that $\Pc \oplus \Qc$ possesses the integer decomposition property. 

On the other hand, suppose that $\Pc \oplus \Qc$ possesses the integer decomposition property. 
Then it is easy to see that each of $\Pc$ and $\Qc$ possesses the integer decomposition property. 
Moreover, since $\Pc \oplus \Qc \subset \RR^{d+e}$ satisfies \eqref{condition}, 
the equality $\delta(\Pc \oplus \Qc)=\delta(\Pc)\delta(\Qc)$ holds by Lemma \ref{Sapporo}. 
Therefore, by Theorem \ref{facets}, either $\Pc$ or $\Qc$ satisfies the condition on its facets 
described in Theorem \ref{Stockholm}, as required. 
\end{proof}

\end{document}